\documentclass{aptpub2}

\def	\Var	{\mathop{\rm Var}}
\def	\Cov	{\mathop{\rm Cov}}

\authornames{Radoslav Marinov}
\shorttitle{Counting Vertices in a Voter-type Model}

\begin{document}


\title{Counting Vertices in a Voter-type Model via Stein's Method}

\author[Concordia University College of Alberta]{Radoslav Marinov}
\address{Concordia University College of Alberta, 7128 Ada Boulevard, Edmonton AB T5B 4E4, Canada \\
radoslav.marinov@concordia.ab.ca}

\begin{abstract}

The Neighborhood Attack model is a Voter type model, which takes a finite graph, assigns $1$'s and $-1$'s to its nodes (vertices), and then runs a Markov chain on the graph by uniformly at random picking a node at every turn, and then switching the values of the node and its neighbors to $1$'s or $-1$'s according to a (not necessarily fair) coin toss. We show, via a Stein's method argument, that for certain (highly symmetric) families of graphs the number of 1's in the Neighbourhood Attack Voter-type model is asymptotically normally distributed as the number of nodes tends to infinity.

\end{abstract}



\section{Introduction and background}


In this paper, we seek to apply Stein's method -- a technique for obtaining convergence (often CLT-type) results for random variables -- on a vertex-count in the Neighborhood Attack Voter-type model.

Voter models are interacting-particle-system models on finite graphs. The original Voter model (introduced independently in the 1970s by Clifford and Sudbury in 1973, and by Holley and Liggett in 1975, as mentioned in \cite{Liggett2005}) can be formulated as follows: Take a connected, $r$-regular (each vertex has $r$ edges) graph of size $n$. Assign $1$'s and $-1$'s to the nodes of the graph. Run a Markov chain on the graph with the following transition procedure: each turn, pick a node at random (under some distribution; usually we take the uniform), pick one of its neighbors at random (usually uniformly), and switch the value of the selected neighbor-node to the value of the originally selected node. Under uniformity of node and neighbor selection, this chain converges to one of two absorbing states, in which all nodes have the same values.

The ``Anti-voter'' model, introduced in \cite{Matloff1977}, has the selected neighbour node adopt a value opposite to that of the originally selected node. Under uniformity (again, of node and neighbor selection) the resulting chain has a stationary distribution.

Persi Diaconis and Christos Athanasiadis in \cite{AthaDiac2010} proposed the following variation of the Voter model: upon selecting a node, instead of picking one of its neighbors, flip a coin (with weight $p$, perhaps taken to be a half), and, according to the result of the cointoss, assign either $1$ or $-1$ to the selected nodes and all its neighbors. The model has been labeled the ``Neighborhood Attack'' model.

Stein's method (first introduced in \cite{Stein1972}) provides an infrastructure for the estimation of the distances between certain classes of random variables and certain (usually classical) distributions, most notably the Gaussian and the Poisson distributions. For practical purposes, we can break Stein's method into three key steps: First, one has to use Stein's identities to establish a bound on the distance between a class of random variables and a specific distribution expected to be close to the given class; second, one has to satisfy the conditions generated in the preceding step; and third, one has to evaluate the acquired bound. The last step typically involves something along the lines of reducing an expression involving a function of the variance of the given random variable.

In \cite{RR1997}, Yosef Rinott and Vladimir Rotar show, using a Stein's method argument, that the sum of the values of the nodes in the Anti-voter model at stationarity is asymptotically normally distributed. The problem Rinnott and Rotar tackled was posed by Aldous and Fill in a book that touches on Voter models, \cite{AldousFillDraftbook}. Our goal in the present article is to show that the sum of the values of the nodes in the Neighborhood Attack model is asymptotically normally distributed, using Stein's method techniques different from the ones employed by Rinott and Rotar.


For an application of the Stein technique in a different context, see the paper \cite{Fulman2004}, in which Jason Fulman shows that the number of descents or inversions in permutations complies to a central limit theorem. Both the current problem and the one examined in \cite{Fulman2004} can be viewed as random walks on hyperplanes; and hence there is a structural similarity between the approach adopted here, and the one in \cite{Fulman2004}.

For more results on the Neighbourhood Attack model, see \cite{AthaDiac2010}, \cite{ChungGra2012}. The former paper introduces the model and presents some results on random walks on hyperplane arrangements. The latter paper studies some properties of the distributions of the implicit Markov chains in models similar to the Neighbourhood Attack model.

For more on Stein's method, see \cite{ChenGoldShao2011}, \cite{Barbour1992}, \cite{Ross2011}. The first two books provide a comprehensive overview of Stein's method in regard to its applications to Normal and Poisson approximations reflexively. The monograph \cite{Ross2011} is an up-to-date survey of Stein's method literature and a useful entry-level source on the subject.


In Section \ref{SectionProblemApproach}, we pose our problem. In Section \ref{SectionSM}, we conduct a brief overview of our main technique: Stein's method. In Section \ref{SectionResults}, we introduce a few definitions and assumptions, and then list the main result of the paper. In the Section \ref{SectionProof}, we provide calculations and proofs for the result. Section \ref{SectionConsequences} interprets the result with some examples of its applicability. We draw conclusions in Section \ref{SectionConclusions}.

\section{Problem and Approach}\label{SectionProblemApproach}
 
We apply the Neighbourhood Attack model (introduced in \cite{AthaDiac2010}) on a given family of (finite) graphs. Randomly assign either $1$ or $-1$ to each node of the graph. As mentioned above, the model does the following each turn:
\begin{itemize}
\item Selects a node uniformly at random.
\item Turns the node and all its immediate neighbours into $1$'s or $-1$'s according to a Bernoulli($p$) distribution with $0<p<1$; we want $p=1/2$ for the sake of symmetry.
\end{itemize}
Given: 1) a connected graph; 2) positive probability of selection for all nodes; and 3) positive probabilities of turning into $1$ or $-1$ for the selected node and its neighbours, the underlying Markov chain, the states of which are the possible permutations of $1$'s and $-1$'s, is irreducible and everywhere recurrent on an essential class of its state space, and therefore possesses a stationary distribution. Assume the considered Markov chain begins at this stationary distribution.

Let $X$ be the number of $1$'s at stationarity. Then $N-X$ equals the number of $-1$'s, where $N$ is the number of nodes.

We want to use Stein's method to show that
\begin{equation*}
\frac{X - \mathbb{E} X}{\sigma_{X}} \xrightarrow[N\rightarrow\infty]{} \mathcal{Z}
\end{equation*}
where $\mathcal{Z}$ is of the standard normal distribution, and $\mathbb{E}X$ and $\sigma_{X}$ are the expectation and standard deviation of $X$.

We derive our result under an assumption of $r$-regularity for the underlying graphs.

We seek to apply Stein's method, and in particular we want to use a result along the lines of Theorem 1.2 in \cite{RR1997}:
\begin{theorem}\label{Theorem1}
Let $(W,W')$ be exchangeable with $\mathbb{E}W = 0$ and $\mathbb{E}W^{2} = 1$. Define the r.v. $R=R(W)$ by
\begin{equation}\label{SteinPairEq}
\mathbb{E}(W'|W) = (1-\lambda)W + R,
\end{equation}
where $0<\lambda<1$. Then, if there is some $A$ for which $|W'-W|\leq A$, we have 
\begin{align*}
\delta := & \sup\{ |\mathbb{E}h(W) - \Phi h|: h\in \mathcal{H}\} \leq \\
& \leq \frac{12}{\lambda} \sqrt{\Var\{ \mathbb{E} \left[ (W'-W)^{2} | W \right] \}} + 37 \frac{\sqrt{\mathbb{E}R^{2}}}{\lambda} + 48\frac{aA^{3}}{\lambda} + 8\frac{aA^{2}}{\sqrt{\lambda}},
\end{align*}
where $\mathcal{H}$ is such that all functions in it are uniformly bounded in absolute value by 1, for any real numbers $c$ and $d$ and any $h\in \mathcal{H}$, the function $h(cx+d)$ is in $\mathcal{H}$, and for any $\epsilon > 0$ and any $h\in\mathcal{H}$, the functions $h_{\epsilon}^{+}, h_{\epsilon}^{-}$ are also in $\mathcal{H}$, and 
\begin{equation*}
\int \tilde{h}(x; \epsilon) \Phi(dx) \leq a\epsilon
\end{equation*}
for some constant $a$ which depends only on the class $\mathcal{H}$.
\end{theorem}

Our $W$ would be some normalization of a vertex-count on the Voter-type model graphs we deal with.

\section{Brief overview of Stein's method}\label{SectionSM}

Stein's technique goes as follows: for a given probability distribution, one can come up with an appropriate operator which implicitly defines the distribution. For example, the operator $A$ in $Af(x)=f'(x)-xf(x)$ implicitly defines the Gaussian distribution, in the sense that 1) $\mathbb{E} Af(Z)=0$ for all absolutely continuous $f$ with $\mathbb{E}|f'(Z)|<\infty$, where $Z$ is a variable with the standard normal distribution; and 2) if for some random variable $W$ we have $\mathbb{E}Af(W)=0$ for all absolutely continuous functions $f$ with $|f'|<\infty$, then $W$ has the standard normal distribution.

Next, for an appropriately chosen $A$, one can solve the differential equation given by 
\begin{equation}\label{eq:SteinDiffEq}
Af(x) = 1_{w\leq x}-\Phi(x),
\end{equation} 
where $\Phi(x)$ is the c.d.f. of the target distribution.

But now, armed with the solution to equation (\ref{eq:SteinDiffEq}), and within the context of an appropriate metric (above we used the Kolomogorov metric), we can produce a bound on the distance $|P(W\leq x)-\Phi(x)|$ between a given distribution we want to analyze, and the target distribution with c.d.f. $\Phi(x)$.

For example,
\begin{equation*}
f_{x}(w) = e^{w^{2}/2}\int_{w}^{\infty}e^{-t^{2}/2} \left( \Phi(x)-1_{t\leq x} \right) dt
\end{equation*}
is the unique bounded solution to
\begin{equation*}
f'_{x}(w)-wf_{x}(w)=1_{w\leq x}-\Phi(x),
\end{equation*}
where $\Phi(x)$ is the c.d.f. of the standard normal. And next, under the Wasserstein metric given by $\mathcal{H}=\{h: \mathbb{R}\rightarrow \mathbb{R}: |h(x)-h(y)|\leq |x-y|\}$, one can show that (for example, see \cite[3.1]{Ross2011}) 
\begin{equation*}
d_{W}(W,Z)\leq (A+B)n^{-1/2}, \quad A=\mathbb{E}|X_{1}|^{3}, \quad B=\frac{\sqrt{2\mathbb{E}[X_{1}^{4}]}}{\sqrt{\pi}},
\end{equation*}
where $W$ is a normalized sum of $n$ i.i.d. standard normal variables endowed with a fourth moment, and $d_{W}(W,Z)$ stands for the Wasserstein distance between $W$ and the standard normal distribution.

The potential utility of Stein's technique in producing powerful bounds and obtaining convergence results is clear; and, indeed, Stein's method has been instrumental in the proofs of a variety of interesting convergence and bounding results. In general, there are two standard avenues of research focusing on Stein's method -- one can try to obtain formulas for bounds on the distances between various target distributions and various random variables (or rather, their distributions) -- examples of recent results in this direction include \cite{FulmanRoss2012} (Exponential distribution), \cite{PikeRen2012} (Laplace), and \cite{GoldIslak2013} (zero-bias couplings and concentration inequalities); and one can use these formulas and techniques to obtain results pertaining to specific problems, including many classic problems such as the Birthday Problem or the Coupon Collector Problem -- for examples, refer to \cite{ChatDiacMeck2005} (comprehensive survey) and \cite{GoldZhang2011} (Lightbulb process).

\section{Initial setup and main result}\label{SectionResults}

\subsection{Initial setup} 

We first seek to show that (\ref{SteinPairEq}) holds. To that end, let $X$ be the number of 1's at stationarity. Let $$Y = 2X - N = \sum_{i=1}^{N} \xi_{i}.$$
Here $N$ is the total number of nodes and $\xi_{i}$ is the value of node $i$ (under an arbitrary indexing). Examining $Y$ is equivalent to examining $X$. Next, define
\begin{equation*}
W:=\frac{Y-\mathbb{E}Y} {\sigma_{Y}}.
\end{equation*}
Note $\sigma_{Y}$ is a constant dependent on $N$:
\begin{equation*}
\sigma^{2}_{Y}=\Var\sum_{i=1}^{N}\xi_{i}=\sum_{i=1}^{N} \Var\xi_{i}+2\sum_{1\leq i<j\leq N} \Cov(\xi_{i},\xi_{j})
\end{equation*}
Now, $W$ is mean-0 variance-1. To get the condition for Theorem \ref{Theorem1}, we first need to define a $W'$ as the equivalent of $W$ after one further turn of the Neighborhood Attack model. That is to say, if $W$ is the normalized node count of the model at some turn of its evolution in stationarity, then $W'$ is the same normalized node count in the next turn.

Note, once again, that we assume $r$-regularity for the graph (i.e. every node has exactly $r$ neighbours). We want $r$-regularity for the sake of symmetry, because without symmetry, the problem under consideration is far less tractable.


\subsection{Main result}

\begin{theorem}\label{MainTheorem}
Under the assumptions 
\begin{equation} \label{rstarAssumption}
r^{*}_{i} = r^{*} \hspace{1cm} \forall i \in I,
\end{equation} 
where $I$ is the index-set of the nodes, $r^{*}_{i}$ is the number of first or second order neighbors node $i$ has, and $r^{*}$ is some constant dependent on the graph; 
and 
\begin{equation}\label{CovCondition}
\Cov(\eta,\theta) \leq 0,
\end{equation} 
where $\eta$ is the count of pairs of neighbors or near-neighbors with values both equal to 1, and $\theta$ is the count of pairs of $-1$'s,
we derive the bound on the distance between (the distributions of) $W$ and the standard normal,
\begin{equation}\label{BoundFINAL}
\delta \leq 48 \frac{r^{2}}{\sqrt{r+1}\sqrt{N}} + (2^{19/2}+48)\frac{\sqrt{r+1}}{\sqrt{N}}.
\end{equation}
\end{theorem}
We first establish bounds on $\sigma_{Y}$ in Section \ref{SectionVarianceY}, and then complete the proof in Section \ref{SectionBigVar1}.


\section{Details and proof}\label{SectionProof}

\subsection{Proving $\mathbb{E}(W'|W)=(1-\lambda)W$}\label{SectionLambda}

Given the $r$-regularity assumption, the sum $\left( \sum_{i=1}^{N} \xi_{i} \right)$ changes each turn by between $-2(r+1)$ and $2(r+1)$. A basic example of a graph of this type is the circle (2-regular) graph, in which we have a set of nodes arranged in a circle, each node with two neighbours.

We also assumed uniformity in choosing nodes and in flipping $1$'s or $-1$'s. Under such conditions $\mathbb{E} \xi_{i}=0$, i.e. at stationarity each node is $1$ or $-1$ with equal probability. The sum of the node values will tend toward 0 (under certain conditions; one of which, clearly, has to do with the number of neighbors each node has, since our model takes only the extreme values over the complete graph), since if nodes of a certain value ($+1$ or $-1$) dominate the graph, we are less likely to see an increase in the number of the nodes of that value.

We show in Section \ref{SectionVarianceY} that, as desired,
\begin{equation*}
\mathbb{E}(W'|W) = \left( 1-\frac{r+1}{N} \right) W,
\end{equation*}
which complies with the Stein linearity condition (\ref{SteinPairEq})
\begin{equation*}
\mathbb{E}(W'|W) = \left( 1-\lambda \right) W + R,
\end{equation*}
where $0\leq\lambda\leq 1$ and $R$ is a random variable. In our case conveniently $R=0$. As for lambda,
\begin{equation*}
\lambda=\frac{r+1}{N}.
\end{equation*}

\subsection{R\"{o}llin's result} 
In general, the next step is to show that $W$ and $W'$ are exchangeable, i.e. $(W,W')=^{d}(W',W)$, as was done in \cite{Fulman2004}. Exchangeability clearly holds when the Markov chain underlying $W$ and $W'$ is stationary and reversible. Reversibility is not always available or easily proved. 

For example, our chain is clearly not necessarily reversible. Consider the circle graph. It is easy to see that for $N$ large, $Y$ can take the value $N-2$ -- i.e. there is an attainable at stationarity arrangement of values for the nodes in which all nodes but one have the value of 1. Now, the probability of going from that arrangement to the all 1's arrangement for which $Y=N$ is positive; but the probability of going from $Y=N$ to $Y=N-2$ is zero, and hence our chain fails to satisfy the detailed balance equations $\pi(x)P(x,y) = \pi(y)P(y,x)$.


However, a recent result by Adrian R\"{o}llin removes the necessity for exchangeability. R\"{o}llin's theorem  (see \cite[Theorem 2.1]{Rollin2008}) states:
\begin{theorem}\label{Rollin}  Assume $W,W'$ are r.v.s on the same probability space, s.t. $\mathcal{L}(W')=\mathcal{L}(W)$ ($\mathcal{L}$ for 'law'), $\mathbb{E}W=0$, $\Var(W)=1$. Given $\mathbb{E}(W'|W)=(1-\lambda)W+R$, for
\begin{equation*} 
\delta:=\sup_{h\in \mathcal{H}}|\mathbb{E}h(W) - \mathbb{E}h( \mathcal{Z} )|
\end{equation*}
(here $\mathcal{Z}$ is the standard normal distribution, and $\mathcal{H}$ is the family of functions associated with the Wasserstein distance), we have
\begin{equation*}
\delta\leq \frac{6}{\lambda}\sqrt{\Var \mathbb{E}^{W}(W'-W)^{2}}+\frac{19\sqrt{\mathbb{E}R^{2}}}{\lambda} + 4\sqrt{\frac{a \mathbb{E}|W'-W|^{3}}{\lambda}}.
\end{equation*}
If also there exists a constant $A$ s.t. $|W'-W|\leq A$ a.s., we have
\begin{equation*}
\delta\leq \frac{12}{\lambda}\sqrt{\Var \mathbb{E}^{W}(W'-W)^{2}}+\frac{37\sqrt{\mathbb{E} R^{2}}}{\lambda} + 32\frac{A^{3}}{\lambda}+6\frac{A^{2}}{\sqrt{\lambda}}.
\end{equation*}
\end{theorem}
\begin{proof}
See \cite[Theorem 2.1]{Rollin2008}.
\end{proof}
In our case $R=0$; so the bound is
\begin{equation}\label{RollinBoundA}
\delta\leq \frac{12}{\lambda}\sqrt{\Var \mathbb{E}[(W'-W)^{2}|W]} + 32\frac{A^{3}}{\lambda} + 6\frac{A^{2}}{\sqrt{\lambda}}.
\end{equation}
The next step is to bound $|W'-W|$. Note $|Y'-Y|\leq 2(r+1)$. So $$|W'-W|\leq\frac{2(r+1)}{\sigma_{Y}}=A.$$
Thus, the bound becomes
\begin{equation}\label{BigBound}
\delta\leq \frac{12N}{(r+1)\sigma_{Y}^{2}}\sqrt{\Var \mathbb{E}[(Y'-Y)^{2}|Y]} + 32\frac{8(r+1)^{2}N}{\sigma_{Y}^{3}}  + 6\frac{4(r+1)^{3/2}\sqrt{N}}{\sigma_{Y}^{2}}.
\end{equation}

\subsection{Bounding the variance of $Y$}\label{SectionVarianceY}
In effect, the next goal is to bound the two terms $\Var \left[ \mathbb{E}(Y'-Y)^{2} | Y \right]$ and $\Var(Y') = \Var(Y) =  \sigma_{Y}^{2}$. For an explicit formula for the value of the corresponding $\sigma_{Y}$ in the Anti-voter model, see \cite[Chapter 14]{AldousFillDraftbook}. Another relevant paper dealing with the Anti-voter case can be found in \cite{DonnellyWelsh1984}.

Let us first try to find $\Var{Y}$. Define $\Delta Y = Y' - Y$ to obtain $Y' = Y + \Delta Y$.

By stationarity, it follows that $0 = \Var(Y')-\Var(Y).$ Hence:
\begin{align}
\nonumber 0 &= \mathbb{E} (Y')^{2} - \mathbb{E} Y^{2} = \\
\nonumber &= \mathbb{E} \left[ Y^{2} + 2Y\Delta Y + (\Delta Y)^{2} - Y^{2} \right] = \\
&= 2\mathbb{E} (Y\Delta Y) + \mathbb{E} (\Delta Y)^{2}. \label{VarY}
\end{align}

To continue, we need to obtain a firmer grip on the r.v. $\Delta Y$. It is easy to see that $\Delta Y$ takes values between $-2(r+1)$ and $2(r+1)$, and that the probability distribution of $\Delta Y$ is a function of certain edge and vertex counts on the coloured graph, themselves random variables. Specifically:

Define $q_{i}$ as the number of nodes s.t. the sum of the values at the node and all its neighbors equals $i$. Clearly, $i$ takes integer values (all odd or all even depending on the parity of $r$) between $-(r+1)$ and $(r+1)$. Specifically, if $r$ is odd, $i$ takes the values $-(r+1), -(r-1), ..., -2, 0, 2, ..., (r-1), (r+1)$; and if $r$ is even, $i$ takes the values $-(r+1), -(r-1),..., -1, 1, ..., (r-1), (r+1)$. In each case $i$ takes $(r+2)$ distinct values. Call the set of those values (the possible values $i$ can take) $I$.

Simple counting produces two useful identities involving the $q_{i}$'s: 
\begin{equation*}
\sum_{i\in I} q_{i} = N  \text{  and  } \sum_{i\in I} iq_{i}= (r+1)Y.
\end{equation*}

Now, at each turn of the Neighborhood Attack process we pick a node uniformly at random (i.e. with probability $1/N$), and turn its value and the value of all its neighbors to either $1$ or $-1$ uniformly at random (i.e. with probability $1/2$).

It thus follows that $\Delta Y$ has the (conditional on $\{q_{i}\}$) p.d.f.:
\begin{equation*}
\Delta Y = \begin{cases}
  (r+1)-i & \text{ with probability } q_{i}\frac{1}{2N} \\
  -(r+1)-i & \text{ with probability } q_{i}\frac{1}{2N} \\
\end{cases}
\end{equation*}
So for example, $\Delta Y$ takes the value of $2(r+1) = (r+1) - (-(r+1))$ with probability $q_{r+1}/2N$; and the value $-2(r+1) = -(r+1) - (r+1)$ with probability $q_{-(r+1)}/2N$.

Thus, we have 
\begin{align*}
\mathbb{E} \left[ \Delta Y | Y \right] &= \mathbb{E}\mathbb{E} \left[ \left( \Delta Y | \{ q_{i} \} \right) | Y \right] =\\
&= \mathbb{E} \left( \sum_{i\in I} \left[ (r+1)-i \right] \frac{q_{i}}{2N} + \sum_{i\in I} \left[ -(r+1)-i \right] \frac{q_{i}}{2N} | Y \right) = \\
&= -2\mathbb{E} \left( \sum_{i\in I}i\frac{q_{i}}{2N} |Y \right) = -\frac{(r+1)Y}{N},
\end{align*}
which is what we stated in Section \ref{SectionLambda}. So:
\begin{lemma}
In the Neighborhood Attack model on an $r$-regular graph,
\begin{equation*}
\mathbb{E} \left[ \Delta Y | Y \right] = -\frac{(r+1)Y}{N}.
\end{equation*}
That is, the Stein linearity condition is satisfied with $\lambda = \frac{(r+1)}{N}$ and $R=0$.
\end{lemma}

We continue from (\ref{VarY}):
\begin{align}
\nonumber 0 &= \mathbb{E} (Y')^{2} - \mathbb{E} Y^{2} = \\
\nonumber &= 2\mathbb{E} (Y \Delta Y) + \mathbb{E} (\Delta Y)^{2} = \\
&= -\frac{2(r+1)}{N}\mathbb{E} (Y^{2}) + \mathbb{E}\mathbb{E} \left( \sum_{i\in I} \left[ (r+1-i)^{2}+(-(r+1)-i)^{2} \right] \frac{q_{i}}{2N} | \{q_{i}\} \right) \label{VarY2}
\end{align}
Let us focus on the $\mathbb{E}\left( \sum_{i\in I} \left[ (r+1-i)^{2}+(-(r+1)-i)^{2} \right] \frac{q_{i}}{2N} | \{q_{i}\} \right)$ term:
\begin{align}
\nonumber & \mathbb{E} \left( \sum_{i\in I} \left[ (r+1-i)^{2}+(-(r+1)-i)^{2} \right] \frac{q_{i}}{2N} | \{q_{i}\} \right) = \\
&= (r+1)^{2} + \frac{1}{N} \left[ \sum_{i\in I}i^{2}q_{i} \right] \leq \label{DeltaSquare} \\
\nonumber &\leq (r+1)^{2} + \frac{1}{N} \left[ \sum_{i\in I}(r+1)^{2}q_{i} \right] = (r+1)^{2} + (r+1)^{2} = 2(r+1)^{2}
\end{align}
And therefore, continuing from (\ref{VarY2}),
\begin{align*}
0 &= \mathbb{E} (Y')^{2} - \mathbb{E} Y^{2} = \\
&= -\frac{2(r+1)}{N}\mathbb{E} (Y^{2}) + \mathbb{E}\mathbb{E} \left( \sum_{i\in I} \left[ (r+1-i)^{2}+(-(r+1)-i)^{2} \right] \frac{q_{i}}{2N} | \{q_{i}\} \right) \leq \\
&\leq -\frac{2(r+1)}{N}\mathbb{E} (Y^{2}) + 2(r+1)^{2},
\end{align*}
meaning
\begin{align*}
\sigma_{Y}^{2} = \Var(Y) = \mathbb{E} (Y^{2}) \leq (r+1)N.
\end{align*}


However, since the $\sigma_{Y}$ terms appear in the denominators of the terms in (\ref{BigBound}), we need either a lower bound of $\sigma_{Y}$ or the exact variance of $Y$.

Observe that we have:
\begin{align*}
\Var(Y) &= \mathbb{E}Y^{2} = \frac{N}{2(r+1)} \mathbb{E}\mathbb{E} \left( \sum_{i\in I} \left[ (r+1-i)^{2}+(-(r+1)-i)^{2} \right] \frac{q_{i}}{2N} | \{q_{i}\} \right) =\\
&= \frac{N}{2(r+1)} \left[ (r+1)^{2} + \frac{1}{N}\mathbb{E} \left( \sum_{i\in I} i^{2}q_{i} \right) \right] \geq \frac{(r+1)N}{2}.
\end{align*}
Thus
\begin{lemma}
For the Neighborhood Attack model on an $r$-regular graph, for $Y=\sum_{i=1}^{N} \xi_{i}$ the sum of the values of the nodes of the graph,
\begin{equation*}
\frac{(r+1)N}{2} \leq \sigma_{Y}^{2} \leq (r+1)N.
\end{equation*}
\end{lemma}

\subsection{Reducing and bounding $\Var \mathbb{E}[(Y'-Y)^{2}|Y]$}\label{SectionBigVar1}
Now we have to evaluate or bound $\Var \mathbb{E}[(Y'-Y)^{2}|Y] = \Var \mathbb{E}^{Y}[(Y'-Y)^{2}]$, as in \cite{RR1997} and \cite{Fulman2004}.

Let us consider the following:
\begin{align*}
\Var &\left[ \mathbb{E}(Y'-Y)^{2} | Y \right] = \Var \left[ \mathbb{E}(\Delta Y)^{2}| Y \right] \leq \Var \left[ \mathbb{E}(\Delta Y)^{2}| \{q_{i}\} \right] =\\
& = \Var \left( (r+1)^{2} + \frac{1}{N} \left[ \sum_{i\in I}i^{2}q_{i}  \right] \right) = \frac{1}{N^{2}}\Var \left( \sum_{i\in I}i^{2}q_{i} \right)
\end{align*}
The transition between the lines follows from (\ref{DeltaSquare}) and (\ref{VarY2}). Also,
\begin{equation*}
\Var \left[ \mathbb{E}(Y'-Y)^{2} | Y \right] = \frac{1}{N^{2}}\Var \left( \mathbb{E} \left( \sum_{i\in I}i^{2}q_{i} \right) | \{ \xi_{i} \} \right) = \frac{1}{N^{2}} \Var\left( \sum_{k=1}^{N} \left( \sum_{j\in \mathcal{N}_{k}}\xi_{j} \right)^{2} \right),
\end{equation*}
where $\mathcal{N}_{k}$ is the set of node $k$ and all its neighbors.

Next,
\begin{align*}
& \Var \left( \sum_{k=1}^{N} \left( \sum_{j\in N_{k}}\xi_{j} \right)^{2} \right) = \\
&= \Var \left( N(r+1) + \sum_{0\leq i,j \leq N; i,j: d(i,j)=1,2} \xi_{i}\xi_{j} \right) = \\
&= \Var \left( 2(\alpha - \beta) \right) = \Var \left( 4(\eta + \theta) \right).
\end{align*}
Here $d(i,j)$ is the distance between nodes $i$ and $j$. For the last line, observe that $N(r+1)$ is invariant, and that the sum $\sum_{0\leq i,j \leq N; i,j: d(i,j)=1,2} \xi_{i}\xi_{j}$ can be interpreted as a sort of an edge count over our graph, with each pair of neighbors or near neighbors of the same sign participating as a $+1$, and each pair of opposite values participating as a $-1$. Each such pair gets counted twice.

Thus let $r^{*}$ be the number of neighbors or near-neighbors (meaning nodes at distances one or two) each node has; $\alpha$ be the number of pairs of neighbors or near-neighbors with equal node-values; and $\beta$ be the number of pairs with opposite node-values. Moreover, $\eta$ corresponds to the count of pairs of neighbors or near-neighbors with values both equal to 1, and $\theta$ is the count of pairs of $-1$'s.

Now, suppose, with $I$ the indexed set of nodes on our graph, that
\begin{equation*}
r^{*}_{i} = r^{*} \hspace{1cm} \forall i \in I.
\end{equation*}
That is, we assume $r^{*}$ is some fixed quantity: i.e. the underlying graph possesses sufficient symmetry so that each node has the same number of neighbors or near-neighbors. For example each node in the circle-graph has 4 neighbors/ near-neighbors.

The assumption that $r_{i}^{*}$ is fixed for all $i$ is not particularly gratuitous, since, either way, $r_{i}^{*}\leq r_{i}^{2}$, and under the present assumptions, $r_{i}=r$ is fixed.

Next, since $2(\alpha+\beta)=r^{*}N$, we have $2(\alpha-\beta)=2\alpha-(r^{*}N-2\alpha)=4\alpha-r^{*}N$. Therefore, $\Var(2(\alpha-\beta)) = \Var(4\alpha) = \Var(4(\eta+\theta))$, where $\eta$ is the number of pairs of neighbors and near neighbors with node-values 1, and $\theta$ is the count of pairs with values $-1$.

On the other hand, we have $2(\eta-\theta)=r^{*}Y$, and therefore $\Var(\eta-\theta) = \frac{(r^{*})^2}{4}\Var(Y)$.

One naturally wonders if we can use the bound we established for $\Var(Y)$ to bound $\Var(\eta+\theta)$. From the definition,
\begin{equation*}
\Var(\eta + \theta) = \Var(\eta)+2\Cov(\eta , \theta)+\Var(\theta)
\end{equation*}
\begin{equation*}
\Var(\eta - \theta) = \Var(\eta) - 2\Cov(\eta , \theta) + \Var(\theta)
\end{equation*}
It suffices to show that $\Cov(\eta , \theta)\leq 0$ to obtain $\Var(\eta + \theta) \leq \Var(\eta - \theta)$. For $\eta$ and $\theta$ to be negatively correlated, an increase in one would have to imply a decrease in the other -- meaning, in our setting, that an increase in the number of edges (i.e. pairs of nodes at distance 1) with ones at both ends would have to imply a decrease in the edges with negative ones at both ends -- and vice versa.

One is tempted to try to use the FKG inequality to prove $\Cov(\eta , \theta)\leq 0$. Specifically, we know that the lattice $\{-1,1\}^{\Gamma}$ (where $\Gamma$ is our graph) is a poset; and that $\eta= f(\{\xi_{i}\})$ is an increasing function of that lattice, while $\theta = g(\{ \xi_{i}\})$ is a decreasing function on the same lattice. Moreover, if we take an element $x$ from $\{-1,1\}^{\Gamma}$, and suppose that $\overline{x}$ is the element we obtain by switching all $-1$'s in $x$ to $+1$'s, and all $+1$'s to $-1$'s, then $f(x)=g(\overline{x})$, and the stationary probability $p_{x}$ of state $x$ occurring in our Markov Chain equals the corresponding probability for state $\overline{x}$ -- that is, $p_{x} = p_{\overline{x}}$.

Now, the FKG theorem (after Fortuin, Kasteleyn, and Ginibre \cite{FKG71}) states that for $\textbf{X}$ a finite distributive lattice, and $\mu$ a non-negative function (really a measure) on it, satisfying the ``log-supermodularity condition'' 
\begin{equation}\label{logsupermod}
\mu(x\wedge y)\mu(x\vee y) \geq \mu(x)\mu(y) \hspace{1cm} \forall x,y\in \textbf{X} 
\end{equation}
yields
\begin{equation*}
\left( \sum_{x\in \textbf{X}} f(x)g(x)\mu(x) \right) \left( \sum_{x \in \textbf{X}} \mu(x) \right) \geq \left( \sum_{x\in \textbf {X}} f(x)\mu(x) \right)\left( \sum_{x\in \textbf {X}} g(x)\mu(x) \right)
\end{equation*}
for any two monotonically increasing (or decreasing) on $\textbf{X}$ functions $f$ and $g$; with the inequality reversed if one of $f$ and $g$ is monotonically increasing and other one monotonically decreasing. 

In our particular case, for the considered lattice, stationary distribution over the lattice, and functions $f$ and $g$, having
\begin{equation*}
\left( \sum_{x\in \{-1,1\}^{\Gamma}} f(x)g(x)p_{x} \right) \left( \sum_{x\in \{-1,1\}^{\Gamma}} p_{x} \right) \leq \left( \sum_{x\in \{-1,1\}^{\Gamma}} f(x)p_{x} \right)\left( \sum_{x\in \{-1,1\}^{\Gamma}} g(x)p_{x} \right)
\end{equation*}
would do the job, since 
\begin{equation*}
\mathbb{E}\eta\theta = \sum_{x\in \{-1,1\}^{\Gamma}} f(x)g(x)p_{x},
\end{equation*}
and 
\begin{equation*}
\mathbb{E}\eta = \sum_{x\in \{-1,1\}^{\Gamma}} f(x)p_{x} = \sum_{x\in \{-1,1\}^{\Gamma}} g(x)p_{x} = \mathbb{E}\theta,
\end{equation*}
and 
\begin{equation*}
\mathbb{E}\eta\theta \leq (\mathbb{E}\eta)(\mathbb{E}\theta)
\end{equation*} 
implies exactly $\Cov(\eta,\theta)\leq 0$.

Unfortunately, our example fails to necessarily satisfy the log-supermodularity condition (\ref{logsupermod}).\footnote{Welsh and Donnelly found that the stationary distribution of the chain underlying the Anti-voted model also fails to satisfy the log-superlinearity condition -- see \cite[Section 5]{DonnellyWelsh1984}.} For example, take a circle graph of odd length. The state in which $1$'s and $-1$'s alternate along the entire graph cannot occur at stationarity, and therefore has measure zero in the stationary distribution of our chain. But the state $w_{1}$ in which we have one $1$ at some node $i$ and everything else is $-1$; and the state $w_{2}$ in which node $i$ and its two neighbors are $-1$'s, and the rest of the graph consists of alternating $1$'s and $-1$'s, can both occur. But then $p_{w_{1}}p_{w_{2}} > 0$, while $p_{w_{1}\wedge w_{2}}=0$, meaning that the log-supermodularity condition fails. One can come up with similar examples for other standard families of graphs. So we fail to have the log-supermodularity condition.

Still, the log-supermodularity condition is only sufficient rather than necessary for our desired result. Hence our results might be obtainable via different means. For now, suppose
\begin{equation*}
\Cov(\eta,\theta) \leq 0.
\end{equation*}

Given (\ref{CovCondition}),
\begin{align*}
\Var \left( \sum_{i\in I}i^{2}q_{i} \right) &= \Var \left( 2(\alpha - \beta) \right) = 4\Var \left( 2(\eta + \theta) \right) \leq \\
&\leq 4 \Var \left( 2(\eta - \theta) \right) = 4(r^{*})^{2}\Var(Y) \leq 4(r^{*})^{2} (r+1)N.
\end{align*}
It follows that
\begin{equation*}
\Var[\mathbb{E}(Y'-Y)^{2} | Y] \leq \frac{4 (r^{*})^{2}(r+1)}{N}
\end{equation*}

We thus arrive at the overall bound (\ref{BigBound}):
\begin{align*}
\delta &\leq \frac{12N}{(r+1)\sigma_{Y}^{2}}\sqrt{\Var \mathbb{E}[(Y'-Y)^{2}|Y]} + 32\frac{8(r+1)^{2}N}{\sigma_{Y}^{3}}  + 6\frac{4(r+1)^{3/2}\sqrt{N}}{\sigma_{Y}^{2}} \leq \\
& \leq 48\frac{r^{*}}{\sqrt{r+1}\sqrt{N}} + 2^{19/2}\frac{\sqrt{r+1}}{\sqrt{N}} + 48\frac{\sqrt{r+1}}{\sqrt{N}}
\end{align*}
Thus our overall bound is:
\begin{equation}\label{finalBound}
\delta \leq 48 \frac{r^{*}}{\sqrt{r+1}\sqrt{N}} + (2^{19/2}+48)\frac{\sqrt{r+1}}{\sqrt{N}},
\end{equation}
where $r^{*}$ is a constant dependent on the underlying family of graphs and satisfying $r^{*}\leq r^{2}$.

This completes the proof of Theorem \ref{MainTheorem}, and derives (\ref{BoundFINAL}). The final bound is of $O \left( \frac{r^{*}}{\sqrt{r}}N^{-\frac{1}{2}} \right)$. 

\section{Consequences and explanation of main result}\label{SectionConsequences}

The bound in (\ref{finalBound}) implies that (under stationarity) the normalized sum of values of the nodes of the graph, $W$, goes in law to the standard normal distribution as the size of the graph rises given $\frac{(r^{*})^{2}}{rN}\rightarrow 0$. Note that $r\leq r^{*}\leq r^{2}$.

Let us consider four specific families of graphs.

First, the complete graph, in which $r=N-1$. On the complete graph, $Y = \sigma_{Y}W$ clearly has the uniform binary distribution taking values $\pm N$. Thus it is to no surprise that our bound on the distance to the normal distribution rises to infinity with $N$.

From the other side of the spectrum of regular graphs, we can take the circuit (or circle or simple cycle) graph, in which we have $N$ ordered nodes, each connected to its predecessor and its successor, with node $N$ connected to nodes $N-1$ and $1$. Here $r=2$, and hence $r^{3/2}/N^{1/2}$ goes to 0 as $N$ increases to infinity.

The argument can be extended to circulant\footnote{ A circulant graph is such that we can arbitrarily index its nodes with 0,1,...,$N-1$, in such a way that if the nodes corresponding to two indices $x$ and $y$ are adjacent, then any two nodes indexed by $z$ and $(z-x+y) \mod N$ are adjacent. Here $N$ is the number of nodes and adjacency of two nodes means they are connected by an undirected edge. } graphs: as long as $r$ stays constant as $N$ rises, $Y$ would converge to the normal in distribution.

For a slightly more complicated example, consider the hypercube graph. One can index the nodes of the $n$-dimensional hypercube graph with a string of $n$ zeros and ones, with nodes differing in exactly one digit being neighbors.

It is easy to see that for an $n$-dimensional hypercube, $r=n$, and $N=2^{n}$. Since 
\begin{equation*}
\frac{r^{3/2}}{N^{1/2}} = \frac{n^{3/2}}{2^{n/2}} \xrightarrow[n\rightarrow\infty]{} 0,
\end{equation*}
we can conclude that $W$ goes in law to the standard normal distribution for the hypercube family of graphs.

Finally, consider the (complete) bipartite graph of size $N=2M$, with $M$ a natural number. For this family, $r=M$, and $N=2M$. On such a graph, $Y$ would frequently take values near $-N,0$ and $N$, and hence cannot be expected to go to the normal in distribution. Indeed, we have 
\begin{equation*}
\frac{r^{3/2}}{N^{1/2}} = \frac{M^{3/2}}{(2M)^{1/2}} \xrightarrow[M\rightarrow\infty]{} \infty.
\end{equation*}
The argument can clearly extend to multipartite graphs of a fixed number of partitions.

\section{Conclusions}\label{SectionConclusions}

To sum up, we have shown that, subject to some symmetry assumptions, the normalized sum of the values of the nodes in the Neighborhood Attack model is at a distance of $O \left( r^{3/2} N^{-\frac{1}{2}} \right)$ to the standard normal distribution in the Wasserstein metric. Hence the sum of the nodes is asymptotically normally distributed as the sizes of the underlying graphs increase, provided that $r^{3/2}N^{-1/2}$ goes to zero as $N$ rises to infinity.

Along the way to the result, we also showed that the node-sum $Y$ in the Neighborhood Attack model on an $r$-regular graph satisfies Stein's linearity condition with $\lambda = \frac{r+1}{N}$ and $R=0$; and that $\sigma_{Y}$ satisfies $\frac{(r+1)N}{2} \leq \sigma_{Y}^{2} \leq (r+1)N.$


\ack
This work is based on the author's 2008-2013 graduate research at the University of Southern California under the advisorship of Prof.  Jason Fulman. The author would like to express his gratitude to Prof. Fulman and the USC Department of Mathematics.

\bibliographystyle{apt}

\end{document}